\documentclass[11pt]{amsart}
\usepackage{amssymb,stmaryrd}
\usepackage{mathrsfs}
\usepackage{enumitem}  
\usepackage{xcolor}
\usepackage{booktabs}
\usepackage{fullpage}

\newtheorem{theorem}{Theorem}[section]

\newtheorem{lemma}[theorem]{Lemma}
\newtheorem{example}[theorem]{Example}

\newtheorem{question}[theorem]{Question}
\newtheorem{proposition}[theorem]{Proposition}
\newtheorem{construction}[theorem]{Construction}

{\theoremstyle{definition}
	\newtheorem{definition}[theorem]{Definition}
	\newtheorem{remark}[theorem]{Remark}
	
}

\DeclareMathOperator {\AGammaL}{A\Gamma L}
\DeclareMathOperator {\GammaL}{\Gamma L}
\DeclareMathOperator {\AGL}{AGL}
\DeclareMathOperator {\GL}{GL}
\DeclareMathOperator {\PGL}{PGL}
\DeclareMathOperator {\PSL}{PSL}
\DeclareMathOperator {\PSU}{PSU}
\DeclareMathOperator {\Sym}{Sym}
\DeclareMathOperator {\Alt}{Alt}
\DeclareMathOperator {\Aut}{Aut}
\DeclareMathOperator {\rank}{rank}
\DeclareMathOperator {\C}{C}
\DeclareMathOperator {\D}{D}
\DeclareMathOperator {\N}{N}
\DeclareMathOperator {\LS}{\mathcal{LS}}
\DeclareMathOperator {\R}{\mathcal{R}}

\DeclareMathOperator {\FF}{\mathbb{F}}
\DeclareMathOperator {\ZZ}{\mathbb{Z}}

\begin{document}
\title{Extremely primitive groups and linear spaces}
\author{Melissa Lee, Gabriel Verret}
\address{Department of Mathematics, University of Auckland, Auckland, New Zealand}
\email{melissa.lee@auckland.ac.nz, g.verret@auckland.ac.nz}
\date{\today}

\begin{abstract}
A finite non-regular primitive permutation group $G$ is \textit{extremely primitive} if a point stabiliser acts primitively on each of its nontrivial orbits.  Such groups have been studied for almost a century, finding various applications.
The classification of extremely primitive groups was recently completed by Burness and Lee, who relied on an earlier classification of soluble extremely primitive groups by Mann, Praeger and Seress.  
Unfortunately, there is an inaccuracy in the latter classification. We correct this mistake, and also investigate regular linear spaces which admit groups of automorphisms that are extremely primitive on points.
\end{abstract}

\maketitle

\section{Introduction}
All groups and linear spaces in this paper are assumed to be finite. Let $G$ be a non-regular primitive group. We say that $G$ is \textit{extremely primitive} if a point stabiliser $G_\alpha$ acts primitively on each of its nontrivial orbits.  Some examples of extremely primitive groups include $\Sym_n$ in its natural action and $\PGL_2(q)$ on the projective line over $\FF_q$. 
Extremely primitive groups have been studied for almost a century~\cite{Manning},  and have arisen several other contexts, including in the constructions of the simple sporadic groups $\mathrm{J}_2$ and $\mathrm{HS}$, as well as in the study of transitive permutation groups with a given upper bound on their subdegrees (see \cite{PaPr} for example).

The problem of classifying extremely primitive groups has garnered significant attention in recent years and a final classification has recently been completed following a series of papers \cite{MPS, BL, BPS1, BPS2, BT, BT2}.  
Unfortunately, an oversight was made in the classification of soluble extremely primitive groups in \cite{MPS}. In this paper, the authors prove a set of necessary conditions  for a soluble group to be extremely primitive \cite[Lemma 3.3]{MPS}. They then claim that these conditions are sufficient \cite[Theorem 1.2]{MPS} but this is never proved and turns out to be incorrect. This error is then reproduced in later work that relies on this result, for example \cite[Theorem 1]{BL} and \cite[Theorem 4]{BT}.

The first purpose of this paper is to correct the classification of soluble extremely primitive groups. Recall that for a prime $p$ and positive integer $d$, a \textit{primitive prime divisor} of $p^d-1$ is a prime that divides $p^d-1$ but not $p^i-1$ for each $1\leq i \leq d-1$.

\begin{theorem}
\label{theo:mainEP}
A soluble group $G$ is extremely primitive if and only if  $G=V\rtimes H\leq \AGammaL_1(p^d)$ for a prime $p$ and $d\geq 1$, where $V$ is the natural vector space and $\C_t\rtimes\C_e\cong  H \leq \GammaL_1(p^d)$,  where $t=|H\cap \GL_1(p^d)|$ is a  primitive prime divisor of $p^d-1$ and either $e=1$,  or $e$ is a prime dividing $d$ and  $t=(p^d-1)/(p^{d/e}-1)$.
\end{theorem}

This corrected version of the classification of soluble extremely primitive groups also impacts results which relied on it. For example, see Theorem~\ref{theo:EPAffine} for an updated version of~\cite[Theorem 1]{BL}.

Our next goal is to study linear spaces that admit a group of automorphisms that is extremely primitive on points. We first introduce some terminology. A  {\it linear space} $S=(\mathcal{P},\mathcal{L})$ is a point-line incidence structure with a set $ \mathcal{P}$ of {\it points} and a set $\mathcal{L}$ of {\it lines}, where each line is a subset of $\mathcal{P}$ such that each pair of points is contained in exactly one line. A linear space is \textit{nontrivial} if it has at least two lines and if every line has at least three points. A linear space is \textit{regular} if all its lines have the same size. An {\it automorphism} of $S$ is a permutation of $\mathcal{P}$ which preserves $\mathcal{L}$. The automorphisms of $S$ form its {\it automorphism group} $\Aut(S)$. (One could consider automorphisms  as acting simultaneously on points and lines but, in a nontrivial linear space, the action on points is already faithful and this viewpoint will usually be simpler for us.)

If $S=(\mathcal{P},\mathcal{L})$ and $S'=(\mathcal{P},\mathcal{L}')$ are two linear spaces with the same set of points, we say that $S'$ is a {\it refinement} of $S$ if every line of $S'$ is contained in a line of $S$. (In other words, for every $\ell'\in \mathcal{L}'$, there exists $\ell\in\mathcal{L}$ such that $\ell'\subseteq\ell$.) 

\begin{theorem}\label{Theo:ExtremelyPrimitiveLinearSpaces}
Let $S$ be a nontrivial regular linear space. If $G\leq \Aut(S)$ is extremely primitive, then one of the following holds:
\begin{enumerate}
\item \label{psl:main} $G = \PSL_2(2^{2^n})$ with point stabiliser $\D_{2(2^{2^n}+1)}$, where $2^{2^n}+1$ is a Fermat prime, and $S$ is a refinement of $\LS(G)$;
\item \label{e=2} $G$ is as in Theorem \ref{theo:mainEP} with $e\geq 2$, and $S$ is a refinement of $\LS(G)$;
\item \label{e=1}  $G$ is as in Theorem \ref{theo:mainEP} with $e=1$.
\end{enumerate}
\end{theorem}

In cases (\ref{psl:main}) and (\ref{e=2}) of Theorem~\ref{Theo:ExtremelyPrimitiveLinearSpaces}, there is a unique ``coarsest'' linear space $\LS(G)$ admitting $G$ as a group of automorphisms.  These linear spaces are  defined in Section~\ref{sec:Star} and some of their properties described in Section \ref{examples}. It is also possible to describe the admissible refinements of $\LS(G)$ in a systematic manner, see Proposition~\ref{prop:EveryRef} and  Examples~\ref{Example:refinement} and~\ref{Example:refinement2}. Classifying the linear spaces which arise in case (\ref{e=1}) seems much more difficult. We give some examples and discuss this further in Section \ref{examples}.

\section{Proof of Theorem~\ref{theo:mainEP}}

Let $G$ be a soluble extremely primitive group. As noted in the proof of \cite[Lemma 3.3]{MPS}, $G$ is a soluble $3/2$-transitive group. (Recall that a transitive permutation group $G$ on a set $\Omega$ is called {\it $3/2$-transitive} if all orbits of $G_\omega$ on $\Omega\setminus\{\omega\}$ have the same size, with this size being greater than 1.) Such groups were classified by Passman. Before stating this classification, we require the following definition. For an odd prime $p$ and integer $d\geq 1$, let $\mathscr{G}(p^d)$ be the subgroup of $\AGL_2(p^d)$ containing all translations and whose point stabiliser consists of all diagonal and antidiagonal matrices in $\GL_2(p^d)$ with determinant $\pm 1$.

\begin{theorem}[\cite{Passman1, Passman2}]\label{theo:Passman}
If $G$ is a soluble $3/2$-transitive group, then one of the following holds:
\begin{enumerate}
\item $G$ is a Frobenius group;
\item $G\leq \AGammaL_1(p^d)$ for a prime $p$ and $d\geq 1$;
\item $G=\mathscr{G}(p^d)$ for an odd prime $p$ and $d\geq 1$;
\item $G$ is one of a collection of groups of degree $3^2,5^2, 7^2,11^2,17^2$ or $3^4$.
\end{enumerate}
\end{theorem}

We now examine each of the cases of Theorem~\ref{theo:Passman} and show that $G$ is extremely primitive if and only if it is as in Theorem~\ref{theo:mainEP}. First, if $G$ is as in Theorem~\ref{theo:Passman} (1), then the proof of \cite[Lemma 3.3]{MPS} implies that $G$ is extremely primitive if and only if it is as in Theorem~\ref{theo:mainEP} with $e=1$.  We also use {\sf GAP} \cite{GAP4} to confirm that no group in Theorem~\ref{theo:Passman} (4) is extremely primitive.

%
%

\subsection{$G=\mathscr{G}(p^d)$}
Let $G=\mathscr{G}(p^d)$ as in the preamble of Theorem \ref{theo:Passman}.  Let $u=[0,0]$ and $v=[1,0]$. By definition, $G_u$ consists of all diagonal and antidiagonal matrices in $\GL_2(p^d)$ with determinant $\pm 1$. This is a group of order $4(p^d-1)$. Next, $G_{uv}= \langle \left( \begin{smallmatrix}
1&0\\0&-1 \end{smallmatrix}\right) \rangle$, so $|G_{uv}|= 2$. By Sylow's Theorem, for $G_{uv}$ to be maximal in $G_u$, one must have $p^d-1=1$, that is $p^d=2$, contradicting that fact that $p$ is odd. So $G_{uv}$ is not maximal in $G_u$ and thus $G$ is not extremely primitive.

\subsection{$G\leq \AGammaL_1(p^d)$}

We need the following preliminary result.

\begin{proposition}\label{prop:SameOrbitSize}
Let $p$ be a prime and $d\geq 1$,  let $H\leq \GammaL_1(p^d)$ and let $T=H\cap \GL_1(p^d)$. We have $H=T\rtimes E$ with $T\cong \C_t$, $E\cong \C_e$ and $e$ divides $d$. Let $V\cong \C_p^d$ be the natural vector space for $\GammaL_1(p^d)$. The following are equivalent:
\begin{enumerate}
\item $T$ and $H$ have the same orbits on $V$;
\item \label{SameSize} the orbits of $H$ on $V\setminus\{0\}$ all have the same size;
\item \label{count} $e=1$ or $p^d-1$ divides $t(p^{d/e}-1)$.
\end{enumerate}
\end{proposition}
\begin{proof}
Since $\GL_1(p^d)$ acts regularly on $V\setminus \{0\}$, we can identify this set with $\GL_1(p^d)$. Let $\alpha$ be a generator of $\GL_1(p^d)\cong \C_{p^d-1}$. Let $m=\frac{p^d-1}{t}$ and note that $T=\langle \alpha^m\rangle$ and the orbits of $T$ on $V\setminus\{0\}$ are its cosets, so of the form $\alpha^i\langle \alpha^m\rangle$ for some $i$. In particular, the orbits of $T$ on $V\setminus\{0\}$ all have the same size so $(1) \Longrightarrow (2)$.

Write $E=\langle f\rangle$.  Now $E$ acts as field automorphisms on $V$ and thus $\alpha^f=\alpha^{p^{d/e}}$.  So 
$$(\alpha^{i+jm})^f=\alpha^{(i+jm)(p^{d/e})}=\alpha^{ip^{d/e}+jmp^{d/e}}.$$
This calculation shows that $f$ preserves $\alpha^i\langle \alpha^m\rangle$ if and only if 
\begin{equation*}\label{eq:mod}
ip^{d/e}\equiv i\pmod{m}.
\end{equation*}

This always holds for $i=0$, so $f$ always preserves $T$. In particular, $H$ always has an orbit of size $|T|$ on $V\setminus\{0\}$.  On the other hand, $ip^{d/e}\equiv i\pmod{m}$ holds for every $i$ if and only if  $p^{d/e}\equiv 1\pmod{m}$ if and only if $m$ divides $p^{d/e}-1$.

Now, if all the orbits of $H$ on $V\setminus\{0\}$ all have the same size, they must have size $|T|$ and $m$  must divide $p^{d/e}-1$, so $(2) \Longrightarrow (3)$.

Finally, if $e=1$, then $H=T$, and if $m$  divides $p^{d/e}-1$, then by the above $T$ and $H$ have the same orbits on $V\setminus\{0\}$ and thus also on $V$, so $(3) \Longrightarrow (1)$.
\end{proof}

\begin{remark}
A version of Proposition \ref{prop:SameOrbitSize} with the extra assumption that $H$ is $p$-exceptional appears as \cite[Lemma 2.7]{GLP}.
\end{remark}

We are now ready to classify the extremely primitive groups in this family.

\begin{theorem}
\label{theo:EPorbs}
Let $p$ be a prime and $d\geq 1$,  let $H\leq \GammaL_1(p^d)$ and let $T=H\cap \GL_1(p^d)$. We have $H=T\rtimes E$ with $T\cong \C_t$,  $E\cong \C_e$ and  $e$ divides $d$. Let $V\cong \C_p^d$ be the natural vector space for $\GammaL_1(p^d)$ and let $G=V\rtimes H\leq \AGammaL(1,p^d)$. Then $G$ is extremely primitive if and only if the following conditions hold:
\begin{enumerate}
\item $t$ is a primitive prime divisor of $p^d-1$, and
\item $e=1$ or $p^d-1=t(p^{d/e}-1)$.
\end{enumerate}
Moreover, all of the nontrivial orbits of $H$ have size $t$.
\end{theorem}
\begin{proof}
Note that $\GL_1(p^d)$ is regular on $V\setminus\{0\}$ so $T$ is semiregular on $V\setminus\{0\}$. In particular, it acts regularly on all of its nontrivial orbits and they each have size $t$. A regular group is primitive if and only if it is trivial or has prime order, so we can assume $t$ is 1 or prime.  If $t$ is not a primitive prime divisor of $p^d-1$, then $T \leq \GL_1(p^b)$ for some proper subfield $\FF_{p^b} \subset \FF_{p^d}$ and $\C_p^b<V$ is $H$-invariant hence $G$ is not primitive. We thus assume that $t$ is a primitive prime divisor of $p^d-1$ and show that $G$ is extremely primitive if and only if (2) holds.  Since $t$ is a primitive prime divisor of $p^d-1$, $T$ acts irreducibly on $V$ so $V\rtimes T$ is primitive, and so is $G$. 

We claim that $G$ is extremely primitive if and only $H$ has the same orbits as $T$. Indeed, apart from the trivial orbit of size $1$, an orbit of $H$ must have order $at$ for some $a$. A stabiliser of a point in such an orbit is a subgroup of order $e/a$ of $\C_e$. This is maximal in $H$ if and only if $a=1$. This proves our claim.

By Proposition~\ref{prop:SameOrbitSize}, $H$ and $T$ have the same orbits if and only if $e=1$ or $p^d-1$ divides $t(p^{d/e}-1)$. Now, $t$ is a primitive prime divisor of $p^d-1$, so if $e\geq 2$, then  $t$ does not divide $p^{d/e}-1$ which itself divides $p^d-1$. Since $t$ is prime, it follows that $t$ divides $\frac{p^d-1}{p^{d/e}-1}$ and thus $p^{d/e}-1$ divides $\frac{p^d-1}{t}$ hence $p^d-1=t(p^{d/e}-1)$.
\end{proof}

\begin{remark}\mbox{}
\begin{enumerate}
\item Note that the condition $p^d-1=t(p^{d/e}-1)$ is quite restrictive. For example, if $d=2$, then it is only satisfied when $p=2$, $e=2$ and $t=3$.
\item Moreover, this condition, together with the fact that $t$ is prime, implies that $e$ cannot be composite. Indeed, if $f$ is a divisor of $e$, then $p^{d/f}-1$ is a multiple of $p^{d/e}-1$ and a divisor of $p^d-1$, but $\frac{p^d-1}{p^{d/e}-1}$ is a prime.
\item The converse of the previous remark does not hold. In other words, $e$ being prime is not sufficient to guarantee that $p^d-1=t(p^{d/e}-1)$. The smallest counterexample is $(p,d,t,e)=(5,2,3,2)$.
\end{enumerate}
\end{remark}

\section{Groups with Property $(\star)$ and the linear space $\LS(G)$}\label{sec:Star}

In this section, we define and prove some basic facts about $\LS(G)$, the linear space that appears in  Theorem~\ref{Theo:ExtremelyPrimitiveLinearSpaces} (\ref{e=2}).

\begin{definition}\label{Def:SpecialLinearSpace}
We say that a permutation group $G$ on $\Omega$ has {\it Property $(\star)$} if, for all $u,v,w\in\Omega$ with $u\neq w$,
$$G_{uv}\leq G_w\Longrightarrow G_{uw}\leq G_v.$$
Let $G$ be a group with Property $(\star)$. For $u,v\in\Omega$, let $\Lambda_{uv}=\{w\in\Omega\mid G_{uv}\leq G_w\}$ and let $\LS(G)$ be the point-line incidence structure having $\Omega$ as set of points and $\{\Lambda_{uv}\mid u, v\in\Omega,u\neq v\}$ as set of lines.
\end{definition}

\begin{proposition}
\label{prop:starLS}
If $G\leq \Sym(\Omega)$ has Property $(\star)$, then $\LS(G)$ is a linear space with $G\leq \Aut(\LS(G))$.
\end{proposition}
\begin{proof}
Let $u$ and $v$ be distinct points. Clearly, we have $u,v\in \Lambda_{uv}$. We show that if $w\in\Lambda_{uv}$  and $u\neq w$,  then $\Lambda_{uv}=\Lambda_{uw}$.  Since $w\in\Lambda_{uv}$, we have $G_{uv}\leq G_w$ and then $G_{uw}\leq G_v$ by Property $(\star)$.  Let $x\in \Lambda_{uv}$, then $ G_{uw}=G_u\cap G_{uw}\leq G_u\cap G_v= G_{uv}\leq G_x$
 so that $x\in\Lambda_{uw}$. This shows that $\Lambda_{uv}\subseteq \Lambda_{uw}$.  Similarly, for $y \in \Lambda_{uw}$, we have $ G_{uv}=G_u\cap G_{uv}\leq G_{uw}\leq G_y$ so $y \in \Lambda_{uv}$, and $\Lambda_{uv}=\Lambda_{uw}$ as claimed.


Now, suppose that $u,v\in\Lambda_{ab}$ for some $a\neq b$ and $u\neq v$. We show that $\Lambda_{ab}=\Lambda_{uv}$. This is clear if $\{u,v\}=\{a,b\}$ so we assume this is not the case. Without loss of generality, we may assume that $u\neq a\neq v$. By the previous paragraph, we have $\Lambda_{au}=\Lambda_{ab}=\Lambda_{av}$. So $v\in \Lambda_{au}$ and hence,  by the previous paragraph, $\Lambda_{au}=\Lambda_{uv}$ and hence $\Lambda_{ab}=\Lambda_{uv}$. This shows that $\LS(G)$ is a linear space. The fact that $G\leq \Aut(\LS(G))$ is obvious from the definition.
\end{proof}

Recall that a {\it flag} in a linear space is a pair $(u,\ell)$  such that $u$ is a point, $\ell$ is a line and $u\in\ell$.

\begin{definition}
Let $S$ be a linear space and $G\leq \Aut(S)$. We say that $(S,G)$ is {\it transverse} if, for every flag $(u,\ell)$ of $S$ and every orbit $\Delta$ of $G_u$, we have $|\ell\cap\Delta|\leq 1$. 
\end{definition}

\begin{proposition}
\label{prop:refinement}
Let $S$ be a linear space and $G\leq \Aut(S)$.  If $G$ has Property  $(\star)$ and $(S,G)$ is transverse, then $S$ is a refinement of $\LS(G)$.
\end{proposition}
\begin{proof}
Write $S=(\mathcal{P},\mathcal{L})$, let $\ell\in\mathcal{L}$ and let $u,v\in\ell$, with $u\neq v$. Recall that $\Lambda_{uv}=\{w\in\Omega\mid G_{uv}\leq G_w\}$ is the unique line of $\LS(G)$ containing $u$ and $v$. We must show that $\ell\subseteq\Lambda_{uv}$. Suppose, by contradiction, that $w\in\ell$ but  $w\not\in\Lambda_{uv}$. Since $w\notin \Lambda_{uv}$, $w$ is not fixed by $G_{uv}$, so $|w^{G_{uv}}|\geq 2$. On the other hand, $G_{uv}$ fixes $u$ and $v$ and thus the unique line of $S$ containing them, namely $\ell$. This implies that $w^{G_{uv}}\subseteq \ell$ and thus $|w^{G_u}\cap \ell|\geq 2$, contradicting the hypothesis that $(S,G)$ is transverse.
\end{proof}




We end this section by showing exhibiting a nice family of groups with Property $(\star)$. This will be useful in later sections.

\begin{lemma}\label{lemma:Has*}
Let $p$ be a prime and $d\geq 1$, let $H\leq \GammaL_1(p^d)$ and let $V\cong \C_p^d$ be the natural vector space for $\GammaL_1(p^d)$. Let $G$ be transitive on $V$ with point stabiliser $H$. If the equivalent conditions from Proposition~\ref{prop:SameOrbitSize} are satisfied, then $G$ has Property $(\star)$.
\end{lemma}
\begin{proof}
Let $u,v,w$ be points such that $u\neq w$ and $G_{uv} \leq G_w$. We need to show that $G_{uw}\leq G_v$. This is clear if $u=v$, so we may assume that $u\neq v$. We can also assume without loss of generality that $u=0$ and thus $H=G_u$. We then have $H_v \leq H_w$ and $v,w\in V\setminus\{0\}$. It follows by Proposition~\ref{prop:SameOrbitSize} (\ref{SameSize}) that $|H_v|=|H_w|$ hence $H_w=H_v\leq G_v$, as required.
\end{proof}

\section{Proof of Theorem~\ref{Theo:ExtremelyPrimitiveLinearSpaces}}
\label{prf_EPSL}
We start with a few preliminary results.  Recall that the {\it rank} of a transitive permutation group is the number of orbits of a point stabiliser.

\begin{theorem}[{\cite[Theorem 1.1]{GZ}}]
\label{theo:GZ2}
Let $S$ be a regular linear space. If $G\leq \Aut(S)$ is extremely primitive on $\mathcal{P}$,  then $\rank(G) \geq 3$ with equality only if $S$ if the affine space $\mathrm{AG}_m(3)$ and $G \leq \AGammaL_1(3^m)$.
\end{theorem}

\begin{proposition}[{\cite[Lemma 2.6]{GZ}}]
\label{prop:lineblocks}
Let $S$ be a linear space, let $G\leq\Aut(S)$ be  extremely primitive and let $(u,\ell)$ be a flag of $S$. If $\Delta$ is an orbit of $G_u$, then $|\ell \cap \Delta| \in \{0, 1, |\Delta| \}$.
\end{proposition}



Our starting point is the following theorem.

\begin{theorem}[{\cite[Corollary 1.4]{GZ}}]
\label{theo:GZ}
Let $S$ be a nontrivial regular linear space. If $G\leq \Aut(S)$ is extremely primitive, then one of the following holds:
\begin{enumerate}
\item \label{AffinePrimitive} $G$ is primitive of affine type;
\item \label{exc} $G$ is an almost simple exceptional group of Lie type;
\item \label{PSL} $G = \PSL_2(2^{2^n})$ with point stabiliser $\D_{2(2^{2^n}+1)}$, where $2^{2^n}+1$ is a Fermat prime.
\end{enumerate}
\end{theorem}

For the rest of this section, we will assume the hypothesis of Theorem~\ref{theo:GZ}. We will then deal with each case in the conclusion in turn and classify the linear spaces $S=(\mathcal{P},\mathcal{L})$ that arise. First, we introduce some basic terminology and results which will often be useful. Write $v = |\mathcal{P}|$ and $b = |\mathcal{L}|$. Since $S$ is regular, all its lines have the same size, which we will denote by $k$. It can be shown (see \cite[Lemma 2.1]{CNP}, for example) that there is also a constant number of lines $r$ meeting each point and the following holds:
\[
 r = \frac{v-1}{k-1} \qquad b = \frac{v(v-1)}{k(k-1)} \qquad v\geq k(k-1)+1.
\]
Moreover, it follows by Fisher's inequality that $b\geq v$ and therefore $r\geq k$.

First, let $G$ be as in Theorem \ref{theo:GZ} (\ref{exc}), that is, an extremely primitive almost simple exceptional group of Lie type, and let $H$ be its point stabiliser. By \cite[Theorem 1]{BT} $(G,H)$ is one of $(\mathrm{G}_2(4),\mathrm{J}_2)$ or $(\mathrm{G}_2(4).2,\mathrm{J}_2.2)$. In each case, the corresponding permutation group has rank $3$ and it then follows by Theorem~\ref{theo:GZ2} that no regular linear space arises in this case.

Next, let $G$ be as in Theorem~\ref{theo:GZ} (\ref{PSL}). Write $q=2^{2^n}+1$. By \cite[Proposition 5.3]{BPS1}, the nontrivial subdegrees of $G$ are all $q$. Since the point stabilisers are isomorphic to $\D_{2q}$, it follows that, given two distinct points $u$ and $v$, $|G_{uv}|=2$. In particular, if  $u\neq w$ and $G_{uv} \leq G_w$, then $G_{uw}  = G_{uv}$, since both groups have order $2$. This shows that $G$ satisfies Property $(\star)$.   Now, let $(u,\ell)$ be a flag of $S$ and $\Delta$ be an orbit of $G_u$ that meets $\ell$. By Proposition~\ref{prop:lineblocks}, $|\ell\cap\Delta|\in\{1,q\}$. If $|\ell\cap\Delta|=q$, then $k\geq q+1$ but $v=|\PSL_2(q-1):\D_{2q}|=(q-1)(q-2)/2$, contradicting the fact that $v\geq k(k-1)+1$. It follows that $(S,G)$ is transverse and we can apply Proposition \ref{prop:refinement} to conclude that $S$ is a refinement of $\LS(G)$, as in Theorem~\ref{Theo:ExtremelyPrimitiveLinearSpaces} (\ref{psl:main}).




It remains to deal with the case when $G$ is as in Theorem~\ref{theo:GZ} (\ref{AffinePrimitive}), that is, an extremely primitive group of affine type. These groups were previously classified in~\cite[Theorem 1]{BL} but this classification relied on the incorrect \cite[Theorem 1.2]{MPS}. Here is then an updated classification of these groups.

\begin{theorem}\label{theo:EPAffine}
Let $p$ be a prime and $d\geq 1$, let $H\leq \GL_d(p)$, let $V$ be the natural vector space for $\GL_d(p)$ and let $G=V\rtimes H$. If $G$ is extremely primitive, then
 one of the following holds:
\begin{enumerate}\addtolength{\itemsep}{0.2\baselineskip}
\item\label{soluble} $G$ is as in Theorem \ref{theo:mainEP};
\item \label{natmod} $p=2$ and $H = \PSL_{d}(2)$ with $d \geq 3$, or $H={\rm Sp}_{d}(2)$ with $d \geq 4$;
\item \label{spor} $p=2$ and $(d,H) = (4,\Alt_6)$, $(4, \Alt_7)$, $(6, \PSU_{3}(3))$ or $(6,\PSU_{3}(3).2)$;
\item  \label{spor2} $p=2$ and $(d,H)$ is one of the following:
\[
\begin{array}{llll}
(10,{\rm M}_{12}) & (10,{\rm M}_{22}) & (10,{\rm M}_{22}.2) & (11,{\rm M}_{23}) \\
(11,{\rm M}_{24}) & (22, {\rm Co}_{3}) &  (24,{\rm Co}_{1}) & (2k, \Alt_{2k+1}) \\
(2k, \Sym_{2k+1}) &  (2\ell, \Alt_{2\ell+2}) & (2\ell, \Sym_{2\ell+2}) & (2\ell, \Omega_{2\ell}^{\pm}(2)) \\
(2\ell, {\rm O}_{2\ell}^{\pm}(2)) &  (8, \PSL_{2}(17)) & (8, {\rm Sp}_{6}(2)) & 
\end{array}
\]
where $k \geq 2$ and $\ell \geq 3$.
\end{enumerate}
\end{theorem}

In the rest of this section, we go through the cases in Theorem~\ref{theo:EPAffine} and, for each case, classify regular linear spaces admitting such groups of automorphisms.

\subsection{Groups arising in Theorem \ref{theo:EPAffine} (\ref{natmod}) or (\ref{spor}).}

The groups in Theorem~\ref{theo:EPAffine}  (\ref{natmod})  are $2$-transitive by \cite[Lemma 2.10.5]{KL}, while those in Theorem~\ref{theo:EPAffine}  (\ref{spor}) are found to be $2$-transitive by direct computation. It follows by Theorem~\ref{theo:GZ2} that no regular linear space arises in this case.


\subsection{Groups arising in Theorem \ref{theo:EPAffine} (\ref{spor2}).}
Adopt the notation of Theorem \ref{theo:EPAffine} (\ref{spor2}). Let $u$ be a point such that $H=G_u$.

We first deal with the infinite families of groups.  If $H= \Omega^\pm_{2l}(2)$ or $\mathrm{O}^\pm_{2l}(2)$, then $G$ has rank 3 by \cite[Lemma 2.10.5]{KL} and by Theorem \ref{theo:GZ2}  no regular linear space arises in this case. Now suppose that $(d,H)$ is one of $(2k, \Alt_{2k+1}),
(2k, \Sym_{2k+1}),  (2\ell, \Alt_{2\ell+2})$ or $ (2\ell, \Sym_{2\ell+2})$, and for conciseness,  let $n=2k+1$ or $2\ell+2$ as appropriate.
In these cases, $H=G_u$ is an alternating or symmetric group acting on the \textit{fully deleted permutation module} $U$. That is,  the subspace $V_0$ of $V_n(2)$ spanned by vectors whose entries sum to 0 if $n$ is odd, or the quotient of $V_0$ by the subspace spanned by the all-ones vector if $n$ is even; $H$  acts by permuting the coordinates of elements of $V_0$, with the corresponding induced action on $U$. The orbits of $H$ on $V_0$ are indexed by the weight of the vectors in it, which must be even, so $H$ has $\lfloor n/2 \rfloor +1$ orbits on $V_0$. When $n$ is even, the orbits of weight $x$ and $n-x$ in $V_0$ get identified in $U$, so $H$ has $\lfloor n/4 \rfloor +1$ orbits on $U$ when $n$ is even. Note that the smallest nontrivial orbit has size $\binom{n}{2}$ so, since $n\geq 5$, it follows that the size of nontrivial orbits is larger than the number of orbits. Let $x$ be the number of orbits of $H$ on $U$ and let $\Delta$ be a nontrivial orbit of $H$ on $U$. We have just shown that $|\Delta|\geq x$. It follows that 
\begin{align}\label{counting}
k(k-1)\leq v-1\leq x|\Delta|\leq |\Delta|^2.
\end{align}
Let $\ell$ be a line meeting $u$ and $\Delta$.  By Proposition \ref{prop:lineblocks}, $|\ell \cap \Delta|$ is equal to $1$ or $|\Delta|$. If $|\ell \cap \Delta|=|\Delta|$, then $k \geq |\Delta|+1$, contradicting (\ref{counting}). We may thus assume that $|\ell \cap \Delta| = 1$. This implies that the orbit of $\ell$ under $H$ has length $|\Delta|$. Since $S$ is nontrivial, there must be another nontrivial orbit $\Delta_1$ of $H$ meeting $\ell$. Repeating the argument above, we find that the orbit of $\ell$ under $H$ has length $|\Delta_1|$, so $|\Delta|=|\Delta_1|$. One can apply this to $\Delta$, the smallest nontrivial orbit. As mentioned earlier,  it has size $\binom{n}{2}$ and one can check that it is the only orbit of that size, which contradicts $|\Delta|=|\Delta_1|$.


It remains to consider the sporadic cases. We give the argument for $(d,H) = (10,{\rm M}_{12})$ here; the others are similar. In this case, $v = 2^{10}$. Recall that $k$ must be an integer such that $3\leq k<v$, $k-1$ divides $v-1$ and $k(k-1)$ divides $v(v-1)$. We find that $k$ is one of $4$, $12$ or $32$, and the corresponding values for $r$ are $341$, $93$ and $33$.
Let $\ell$ be a line through $u$.  If $\Delta$ is a nontrivial orbit of $H$ meeting $\ell$, then by Proposition \ref{prop:lineblocks},  $|\ell \cap \Delta|$ is equal to $1$ or $|\Delta|$. By direct computation, we find that $H$ has orbit lengths $1,66,66,396,495$.  If  $|\ell \cap \Delta| =|\Delta|$, then $k\geq 1+|\Delta| \geq 67$, a contradiction.  It follows that  $|\ell \cap \Delta| = 1$ and the orbit of $\ell$ under $H$ has length $|\Delta|$. Repeating this argument for other lines through $u$, we conclude that $r$ is a linear combination of the nontrivial orbit lengths of $H$, which is a contradiction.

\subsection{Groups arising in Theorem \ref{theo:EPAffine} (\ref{soluble}).}\label{sec:what}
It remains to deal with the groups arising in Theorem \ref{theo:EPAffine} (\ref{soluble}). Let $G$, $p$, $d$, $t$, $e$ be as in Theorem \ref{theo:mainEP}.  By Lemma~\ref{lemma:Has*}, $G$ has Property $(\star)$ and it follows by Proposition \ref{prop:starLS} that  $\LS(G)$ is a linear space.

If $e=1$, then Theorem~\ref{Theo:ExtremelyPrimitiveLinearSpaces} (\ref{e=1}) holds. We thus assume that $e\geq 2$. Our next goal is to show that $(S,G)$ is transverse. Let $(u,\ell)$ be a flag of $S$ and $\Delta$ be a nontrivial orbit of $G_u$. By Theorem \ref{theo:EPorbs}, $|\Delta|=t$. By Proposition \ref{prop:lineblocks}, $|\ell\cap\Delta|$ is equal to one of $0$, $1$ or $t$. In view of a contradiction, we can assume that $|\ell\cap\Delta|=t$.  This implies that $k\geq t+1$ and, by Fisher's inequality, $r\geq t+1$. Since $e\geq 2$, we have
\[
p^d-1 =v-1= r(k-1) > t^2 =\left(\frac{p^d-1}{p^{d/e}-1} \right)^2\geq\left( \frac{p^d-1}{p^{d/2}-1}\right)^2 =  \left(p^{d/2}+1\right)^2,
\]
which is a contradiction.  This concludes the proof that $(S,G)$ is transverse. We can now apply Proposition \ref{prop:refinement} to conclude that $S$ is a refinement of $\LS(G)$, as in Theorem~\ref{Theo:ExtremelyPrimitiveLinearSpaces} (\ref{e=2}).

\section{Refinements of line-transitive spaces}\label{sec:ref}

In this section, we give a construction for refinements of line-transitive linear spaces such that the line-transitive group also acts on the refined space. We also show that all such refined spaces arise in this way.

Given a line $\ell$ of a linear space $S=(\mathcal{P},\mathcal{L})$, we write $\mathcal{P}(\ell)$ for the set of points of $S$ incident with $\ell$. (In most of this paper, we simply identify $\ell$ with $\mathcal{P}(\ell)$, but we avoid this in this section to reduce possible confusion.) Given $G\leq \mathrm{Sym}(\Omega)$, we write $G_\ell$ and $G_{[\ell]}$ for the subgroup of $G$ preserving $\mathcal{P}(\ell)$ setwise and pointwise, respectively. We also write $G_\ell^\ell$ for the permutation group induced by $G_\ell$ on $\mathcal{P}(\ell)$.

\begin{construction}\label{consSpace}
The input of the construction is the following:
\begin{enumerate}
\item \label{hyp:1} A linear space $S=(\mathcal{P},\mathcal{L})$  with $G\leq\Aut(S)$ such that $G$ is transitive on $\mathcal{L}$, and a line  $\ell\in \mathcal{L}$.
\item \label{hyp:2} A linear space $T=(\mathcal{P}(\ell),\mathcal{TL})$ such that $G_\ell^\ell\leq \Aut(T)$.
\end{enumerate}
The output of the construction is an incidence structure $R=\R(\mathcal{P},\mathcal{L},G,\ell,\mathcal{TL})$. The set of points of $R$ is $\mathcal{P}$ while the set of lines of $R$ is $\{t^g\mid t\in\mathcal{TL}, g\in G\}$.
\end{construction}

\begin{proposition}\label{prop:OutputIsRef}
Using the notation of Construction~\ref{consSpace}, the output  $R$ of the construction is a linear space which is a refinement of $S$ and with $G\leq\Aut(R)$.
\end{proposition}
\begin{proof}
It is clear from the construction that $G\leq\Aut(R)$. We first show  that $R$ is a linear space. Let $u,v\in\mathcal{P}$, $u\neq v$.  There exists $\ell_{uv}\in\mathcal{L}$ such that $u,v\in \ell_{uv}$. Since $G$ is transitive on $\mathcal{L}$, there exists $g\in G$ such that $\ell_{uv}^g=\ell$, so $u^g,v^g\in\mathcal{P}(\ell)$. Since $T$ is a linear space, there is $t\in\mathcal{TL}$ such that $u^g,v^g\in t$ so $t^{g^{-1}}$ is a line of $R$ containing $u$ and $v$. Now, let $k$ be a line of $R$ containing $u^g$ and $v^g$. By definition, there exists $h\in G$ such that $k^h\in \mathcal{TL}$, so $u^{gh},v^{gh}\in \mathcal{P}(\ell)$. Since $\ell$ is the unique line of $\mathcal{L}$ containing $u^g$ and $v^g$, we have $\ell^h=\ell$,  hence $h\in G_\ell$ and $h^{\ell}\in G_\ell^\ell\leq \Aut(T)$. It follows that  $h^\ell$ preserves $\mathcal{TL}$ and thus $k\in \mathcal{TL}$. Since $T$ is a linear space, it follows that there is a unique line of $R$ containing $u^g$ and $v^g$ (namely $k$), and the same holds for $u$ and $v$. We have shown that  $u,v$ are on a unique line in $R$ so $R$ is a linear space.

We now show that $R$ is a refinement of $S$. Let $k$ be a line of $R$. By definition, there exist $t\in\mathcal{TL}$ and  $g\in G$ such that $k=t^g$. By definition,  $t\subseteq \ell$ hence $k=t^g\subseteq\ell^g\in \mathcal{L}$, as required.
\end{proof}

\begin{proposition}\label{prop:EveryRef}
	Let $S=(\mathcal{P},\mathcal{L})$, $G$ and $\ell$ be as in Construction~\ref{consSpace} (\ref{hyp:1}). If $R$ is a refinement of $S$ with $G\leq \Aut(R)$, then there exists a linear space $T=(\mathcal{P}(\ell),\mathcal{TL})$ with $G_\ell^\ell\leq \Aut(T)$ such that $R=\R(\mathcal{P},\mathcal{L},G,\ell,\mathcal{TL})$.
\end{proposition}
\begin{proof}
Write $R=(\mathcal{P},\mathcal{RL})$ and let $\mathcal{TL}=\{k\in\mathcal{RL}\mid k\subseteq \ell \}$. We first show that $T=(\mathcal{P}(\ell),\mathcal{TL})$ is a linear space. Let $u$ and $v$ be distinct elements of $\mathcal{P}(\ell)$.  These points must be contained in a unique line of $S$, which must necessarily be $\ell$. Moreover, they must also be contained in a unique line of $R$, say $k\in \mathcal{RL}$. Since $R$ is a refinement of $S$, $k\subseteq \ell$ so $k\in\mathcal{TL}$. If $k'$ is a line in $\mathcal{TL}$ containing $u$ and $v$, then by definition $k'\in \mathcal{RL}$ and $k=k'$ since $R$ is a linear space. This shows that $T$ is a linear space. By definition, $G_\ell^\ell\leq\Sym(\mathcal{P}(\ell))$. Since $G\leq \Aut(R)$, it follows that   $G_\ell^\ell\leq \Aut(T)$.

It remains to show that $R=(\mathcal{P},\mathcal{RL})=\R(\mathcal{P},\mathcal{L},G,\ell,\mathcal{TL})$. These have the same set of points so it remains to show that $\mathcal{RL}=\{t^g\mid t\in\mathcal{TL}, g\in G\}$. Let $t\in\mathcal{TL}$ and $g\in G$. By definition, $t\in\mathcal{RL}$ but $G\leq \Aut(R)$, so $t^g\in\mathcal{RL}$. This shows that $\{t^g\mid t\in\mathcal{TL}, g\in G\}\subseteq\mathcal{RL}$. For the other direction, let $t\in \mathcal{RL}$. Since $R$ is a refinement of $S$, there exists $k\in \mathcal{L}$ such that $t\subseteq k$ and, since $G$ is transitive on $\mathcal{L}$, there exists $g$ such that $k^g=\ell$. Now, $t^g\subseteq k^g=\ell$ and $t^g\in \mathcal{RL}$ so by definition $t^g\in  \mathcal{TL}$, as required.
\end{proof}

In light of Propositions~\ref{prop:refinement} and~\ref{prop:EveryRef}, it will be important to be able to determine $G_{\ell}^{\ell}$, especially in the case of $\LS(G)$. This is the content of our next two results, which will be useful in Section \ref{examples}. (For $H\leq G$, we denote the normaliser of $H$ in $G$ by $\N_G(H)$.)

\begin{lemma}
\label{line_aut_structure}
Let $G \leq \mathrm{Sym}(\Omega)$ and $u,v\in \Omega$. If $\ell =\{w\in\Omega\mid G_{uv}\leq G_w\}$, then the following statements hold:
\begin{enumerate}
\item $G_{[\ell]}=G_{uv}$;
\item $G_{\ell} = \N_G(G_{uv})$;
\item $G_{\ell}^{\ell}  \cong \N_G(G_{uv})/G_{uv}$.
\end{enumerate}
\end{lemma}
\begin{proof}
It follows from the definition of $\ell$ that $u,v\in\ell$ and $G_{[\ell]}= G_{uv}$, establishing (1). If $g \in \N_G(G_{uv})$ and $w \in \ell$, then $G_{uv} \leq G_w$ which implies $G_{uv}=(G_{uv})^g \leq G_{w}^g = G_{w^g}$, so $w^g \in \ell$. This shows $\N_G(G_{uv})\leq G_{\ell}$. In the other direction, if $g\in G_{\ell}$, then $(G_{[\ell]})=G_{[\ell]^g}=G_{[\ell]}$ so $g\in\N_G(G_{[\ell]})=\N_G(G_{uv})$. This completes the proof of (2). Finally, by the first isomorphism theorem, $G_{\ell}^{\ell}\cong G_{\ell}/G_{[\ell]}$ and (3) follows.
\end{proof}

Recall that a permutation group is \textit{semiregular} if all its point stabilisers are trivial and \textit{regular} if it is also transitive.

\begin{lemma}\label{lemma:Gll-semiregular}
Let $S$ be a linear space with $G\leq\Aut(S)$ and let $(u,\ell)$ be a flag of $S$. If $(S,G)$ is transverse, then $G_u\cap G_\ell=G_{[\ell]}$ and in particular $G_{\ell}^{\ell}$ is semiregular.
\end{lemma}
\begin{proof}
Clearly $G_{[\ell]}\leq G_u\cap G_\ell$ so we must show that $G_u\cap G_\ell$ fixes every point of $\ell$. It clearly fixes $u$, so let $v$ be another point of $\ell$. Let $\Delta=v^{G_u}$. Since $(S,G)$ is transverse, we have  $\ell\cap\Delta=\{v\}$, but $\ell\cap\Delta$ is preserved by $G_u\cap G_\ell$ so $v$ is fixed. The second statement follows from the first simply by unpacking the definitions involved.
\end{proof}

\section{Questions and examples}
\label{examples}
Our proof of cases (\ref{psl:main}) and (\ref{e=2}) of Theorem~\ref{Theo:ExtremelyPrimitiveLinearSpaces} both involved showing that $G$ has Property $(\star)$, that $(S,G)$ is transverse and then applying Proposition \ref{prop:refinement}. The first part of this approach still works in case (\ref{e=1}), as we showed in Section~\ref{sec:what} that $G$ has Property $(\star)$ even in this case, but there are at least two other issues. First, in case (\ref{e=1}), the stabiliser of two distinct points is trivial, so $\mathcal{L}(G)$ is the trivial linear space with a single line. Knowing that $S$ is a refinement of $\mathcal{L}(G)$ when $(S,G)$ is transverse therefore gives no information. The second problem is that, unlike in cases (\ref{psl:main}) and (\ref{e=2}), $(S,G)$ need not be transverse, as the following examples show.

\begin{example}\label{ex:projplane}
Let $\ell=\{0,1,3,9\}\subseteq\ZZ_{13}$. This is a perfect difference set so its translates form the lines of a linear space $S$ with point-set $\ZZ_{13}$. By definition, $\ZZ_{13}\leq \Aut(S)$. Let $\alpha$ be the permutation of $\ZZ_{13}$ corresponding to ``multiplication by $3$''. Note that $\ell^\alpha=\ell$ which implies that $\alpha\in\Aut(S)$ hence $G=\ZZ_{13}\rtimes\langle\alpha\rangle\leq\Aut(S)$. It is easy to see that $G$ is  extremely primitive on  $\ZZ_{13}$. Let $u=0$ and $\Delta=\{1,3,9\}$. Note that $G_u=\langle\alpha\rangle$  and that $\Delta$ is an orbit of $G_u$ with $|\ell\cap\Delta|=3$ so $(S,G)$ is not transverse. (Note that here $v=13$ and $k=4$ so $S$ is in fact the unique projective plane of order $3$ and its automorphism group is much bigger than $G$.)
\end{example}

Note that, in Example~\ref{ex:projplane}, $G$ is line-transitive (by construction). We were not able to find any other such example and wonder if any exist. More precisely:

\begin{question}\label{question}
Let $S$ be a nontrivial linear space with $G\leq\Aut(S)$ such that $G$ that is extremely primitive on points, transitive on lines and such that $(S,G)$ is not transverse. Does it follow that  $S$ is the projective plane of order $3$?
\end{question}

 As mentioned at the start of this section, in cases (\ref{psl:main}) and (\ref{e=2}) of Theorem~\ref{Theo:ExtremelyPrimitiveLinearSpaces}  $(S,G)$ is transverse so an example for Question~\ref{question} must arise from case (\ref{e=1}), that is $\C_p^d\rtimes\C_t\cong G\leq \AGL(1,p^d)$ for $t$ a primitive prime divisor of $p^d-1$. We have checked using {\sf{GAP}} that there is no other example with fewer than 1000 points. On the other hand, there seems to be plenty of examples which are not line-transitive:

\begin{example}\label{Ex:notLT}
Let $G$ be \texttt{PrimitiveGroup(25,1)} in {\sf{GAP}}. This group is generated by the following two permutations:
\begin{align*}
&(2,19,6)(3,25,11)(4,7,16)(5,13,21)(8,24,9)(10,15,14)(12,17,20)(18,23,22) {~\text and}\\ 
&(1,2,3,5,4)(6,7,8,10,9)(11,12,13,15,14)(16,17,18,20,19)(21,22,23,25,24).
\end{align*}
One can check that $G$ is extremely primitive and $\C_5^2\rtimes\C_3 \cong G \leq \mathrm{AGL}_1(5^2)$. Now, let $\ell_1 = \{ 1, 2, 6, 19\}$,  $\ell_2 = \{ 1, 3, 11, 25 \}$,  $\mathcal{L}=\ell_1 ^G \cup \ell_2^G$ and $S=(\{1,\ldots,25\},\mathcal{L})$. One can check that $S$ is a linear space with $(v,k,r,b) = (25,4,8,50)$. By construction, $G\leq \Aut(S)$ and it turns out that $G$ has two orbits on lines, with representatives $\ell_1$ and $\ell_2$. Note that $\Delta_i=\ell_i\setminus \{1\}$  is an orbit of $G_1$ with $|\Delta_i\cap \ell_i|=3$ and again $(S,G)$ is not transverse. As a final remark, we note that $\Aut(S)$ actually is isomorphic to \texttt{PrimitiveGroup(25,3)} and this group is not extremely primitive, nor transitive on lines.
\end{example}

Examples~\ref{ex:projplane} and~\ref{Ex:notLT} show that the classification of linear spaces arising from Theorem~\ref{Theo:ExtremelyPrimitiveLinearSpaces} (\ref{e=1}) is  more complicated than in cases (\ref{psl:main}) and (\ref{e=2}). This is not that surprising, since in case (\ref{e=1})   $G$  is very ``small'' (its point stabiliser has prime order) so the restriction $G\leq\Aut(S)$ is much weaker.

We now describe  $\LS(G)$ from Theorem~\ref{Theo:ExtremelyPrimitiveLinearSpaces} (\ref{psl:main}) and (\ref{e=2}) in more detail.

\begin{example}\label{ex:Fermat}
Let $q=2^{2^n}+1$ be a Fermat prime and  $G = \PSL_2(q-1)$ with point stabiliser $\D_{2q}$, as in Theorem \ref{Theo:ExtremelyPrimitiveLinearSpaces} (\ref{psl:main}).  As observed in Section \ref{prf_EPSL},  the two-point stabilisers in  $G$ have order $2$; hence we may identify each line $\ell$ of $\mathcal{LS}(G)$ with the involution $g_\ell$ that fixes any two points on it. By identifying points of $\mathcal{LS}(G)$ with their point stabilisers, we see that $\mathcal{LS}(G)$ is the \textit{Witt-Bose-Shrikhande space} $W(q-1)$ \cite[\S2.6]{BDD} with parameters
\[
(v,b,k,r) = \left( \frac{(q-1)(q-2)}{2},q(q-2), \frac{q-1}{2}, q \right)\!.
\]
 By Lemma \ref{line_aut_structure},  $G_\ell=\N_G(\langle g_\ell\rangle)$, which is elementary abelian of order $q-1$, while $G_\ell^\ell\cong\N_G(\langle g_\ell\rangle)/\langle g_\ell\rangle$,  an elementary abelian group of order $\frac{q-1}{2}$, which is semiregular by Lemma~\ref{lemma:Gll-semiregular}. Since $\frac{q-1}{2}=k$, it follows that $G_\ell^\ell$ is in fact regular. By the orbit-stabiliser theorem, the orbit of $\ell$ under $G$ has size $|G|/|G_\ell| = q(q-2) = b$, so $G$ is transitive on lines. 
\end{example}

\begin{example}\label{ex:e=2}
Let $p$ be a prime, $d\geq 2$, $t$ be a primitive prime divisor of $p^d-1$ and $e$ be a prime dividing $d$ such that $t=(p^d-1)/(p^{d/e}-1)$.
Let $G\leq \AGammaL_1(p^d)$ with point stabiliser  $H=\C_t\rtimes\C_e \leq \GammaL_1(p^d)$, as in Theorem \ref{Theo:ExtremelyPrimitiveLinearSpaces} (\ref{e=2}).
A two-point stabiliser in $G$ is conjugate to the group $\C_e$ which acts as field automorphisms on $\mathbb{F}_{p^d}$ and hence fixes the subfield of order $p^{d/e}$, so $\mathcal{LS}(G)$ has parameters
\[
(v,b,k,r) = (p^d,p^{d-d/e}t,p^{d/e}, t).
\]
By Lemma \ref{line_aut_structure}, $G_\ell = \N_G(\C_e) =\C_p^{d/e} \times \C_e$, while $G_\ell^\ell \cong \N_G(\C_e)/\C_e \cong\C_p^{d/e}$. Since $k=p^{d/e}$, we again obtain that $G_\ell^\ell$ is regular. Finally, the orbit of $\ell$ under $G$ has size $|G|/|G_\ell| = p^dte/(p^{d/e}e)= b$, so $G$ is transitive on lines. 
\end{example}

Note that in both previous examples, $G$ is transitive on lines  of $\mathcal{LS}(G)$ so by Proposition~\ref{prop:EveryRef}, all the refinements  of $\mathcal{LS}(G)$  admitting $G$ as a group of automorphisms arise via Construction~\ref{consSpace}.  We now present two final examples, which give nontrivial  such refinements and thus are also examples for Theorem~\ref{Theo:ExtremelyPrimitiveLinearSpaces}.

\begin{example}\label{Example:refinement}
Let $(p,d,e,t)=(7,5,5,2801)$, let $G$ be as in Example~\ref{ex:e=2} and let $S=\mathcal{LS}(G)$. If $\ell$ is a line of $S$, then $|\ell|=7$ and $G_\ell^\ell\cong\C_7$ is regular. If we set $T=(\mathcal{P}(\ell),\mathcal{TL})$ to be the Fano plane, then we have $\C_7\leq\Aut(T)$. By Proposition~\ref{prop:OutputIsRef}, the output $R$ of Construction~\ref{consSpace} is a linear space on $7^5$ points with lines of size $3$ which is a refinement of  $S$ and with $G\leq\Aut(R)$. Since $C_7$ is transitive on $\mathcal{TL}$, it follows that $G$ is transitive on the lines of $R$.
\end{example}

\begin{example}\label{Example:refinement2}
Let $q$ be the Fermat prime $65537=2^{16}+1$, let $G$ be as in Example~\ref{ex:Fermat} and let $S=\mathcal{LS}(G)$. If $\ell$ is a line of $S$, then $|\ell|=2^{15}$ and $G_\ell^\ell\cong\C_2^{15}$ is regular. Let $H\leq \GammaL_1(2^{15})$ such that $|H\cap \GL_1(2^{15})|=1057$ and $H=\C_{1057}\rtimes \C_3$. Let $V\cong \C_2^{15}$ be the natural vector space for $\GammaL_1(2^{15})$ and let $A=V\rtimes H\leq\AGammaL_1(2^{15})$. Note that Proposition~\ref{prop:SameOrbitSize} (\ref{count}) is satisfied (with $(p,d,e,t)=(2,15,3,1057)$), hence Lemma~\ref{lemma:Has*} implies that $A$ has Property $(\star)$ and by Proposition~\ref{prop:starLS}, $T:=\mathcal{LS}(A)$ is a linear space  such that $\C_2^{15}\cong V\leq A\leq \Aut(T)$. By Proposition~\ref{prop:OutputIsRef}, the output $R$ of Construction~\ref{consSpace} is a linear space which is a refinement of $S$ and with $G\leq\Aut(R)$. By similar calculations as in Example \ref{ex:e=2} we deduce that each line  contains $k:=p^{d/e}=32$ points and thus $R$ has parameters
\[
(v,k,r) = (2^{15}(2^{16}-1),32, qt).
\]
Since $(S,G)$ is transverse, so is $(R,G)$. In particular,  if $(u,m)$ is a flag of $R$ and $u\neq v \in \ell$,  then $|m^{G_u}| = |v^{G_u}| =|G_u|/|G_{uv}| = 2q/2=q$. It follows that $G_u$ has $r/q=t=1057$ orbits on lines meeting $u$.  Note that $G$ is smaller than the number of lines of $R$, so $G$ is not transitive on lines of $R$.
\end{example}

Example~\ref{Example:refinement2} answers a question posed implicitly  in \cite{GZ} about the existence of regular linear spaces that admit an extremely primitive automorphism group with classical socle, other than the Witt-Bose-Shrikhande spaces. In \cite[Lemma 3.2]{GZ}, Guan and Zhou show that in that case $G$ is as in Example~\ref{ex:Fermat}, that $q$ is at least $65537$ (the largest known Fermat prime) and that $G_u$ has at least $73$ orbits on lines meeting $u$, but are unable to determine if any examples actually arise. Example~\ref{Example:refinement2} gives one such example and our approach using refinements can be used to construct more.

\bibliographystyle{abbrv}
\bibliography{EPP}

\begin{thebibliography}{10}

\bibitem{BDD}
F.~Buekenhout, A.~Delandtsheer, and J.~Doyen.
\newblock Finite linear spaces with flag-transitive groups.
\newblock {\em J. Combin. Theory Ser. A}, 49(2):268--293, 1988.

\bibitem{BL}
T.~C. {Burness} and M.~{Lee}.
\newblock {On the classification of extremely primitive affine groups}.
\newblock (arXiv:2107.04302), 2021.

\bibitem{BPS1}
T.~C. Burness, C.~E. Praeger, and \'A.~Seress.
\newblock Extremely primitive classical groups.
\newblock {\em J. Pure Appl. Algebra}, 216(7):1580--1610, 2012.

\bibitem{BPS2}
T.~C. Burness, C.~E. Praeger, and \'A.~Seress.
\newblock Extremely primitive sporadic and alternating groups.
\newblock {\em Bull. Lond. Math. Soc.}, 44(6):1147--1154, 2012.

\bibitem{BT2}
T.~C. Burness and A.~R. Thomas.
\newblock A note on extremely primitive affine groups.
\newblock {\em Arch. Math. (Basel)}, 116(2):141--152, 2021.

\bibitem{BT}
T.~C. Burness and A.~R. Thomas.
\newblock The classification of extremely primitive groups.
\newblock {\em Int. Math. Res. Not. IMRN}, (13):10148--10248, 2022.

\bibitem{CNP}
A.~R. Camina, P.~M. Neumann, and C.~E. Praeger.
\newblock Alternating groups acting on finite linear spaces.
\newblock {\em Proc. London Math. Soc. (3)}, 87(1):29--53, 2003.

\bibitem{GAP4}
The GAP~Group.
\newblock {\em {GAP -- Groups, Algorithms, and Programming, Version 4.9.3}},
  2018.

\bibitem{GLP}
M.~Giudici, M.~W. Liebeck, C.~E. Praeger, J.~Saxl, and P.~H. Tiep.
\newblock Arithmetic results on orbits of linear groups.
\newblock {\em Trans. Amer. Math. Soc.}, 368(4):2415--2467, 2016.

\bibitem{GZ}
H.~Guan and S.~Zhou.
\newblock Extremely primitive groups and linear spaces.
\newblock {\em Czechoslovak Math. J.}, 66(141)(2):445--455, 2016.

\bibitem{KL}
P.~Kleidman and M.~Liebeck.
\newblock {\em The subgroup structure of the finite classical groups}, volume
  129 of {\em London Mathematical Society Lecture Note Series}.
\newblock Cambridge University Press, Cambridge, 1990.

\bibitem{MPS}
A.~Mann, C.~E. Praeger, and \'A.~Seress.
\newblock Extremely primitive groups.
\newblock {\em Groups Geom. Dyn.}, 1(4):623--660, 2007.

\bibitem{Manning}
W.~A. Manning.
\newblock Simply transitive primitive groups.
\newblock {\em Trans. Amer. Math. Soc.}, 29(4):815--825, 1927.

\bibitem{PaPr}
D.~V. Pasechnik and C.~E. Praeger.
\newblock On transitive permutation groups with primitive subconstituents.
\newblock {\em Bull. London Math. Soc.}, 31(3):257--268, 1999.

\bibitem{Passman1}
D.~S. Passman.
\newblock Solvable {${3/2}$}-transitive permutation groups.
\newblock {\em J. Algebra}, 7:192--207, 1967.

\bibitem{Passman2}
D.~S. Passman.
\newblock Exceptional {$3/2$}-transitive permutation groups.
\newblock {\em Pacific J. Math.}, 29:669--713, 1969.

\end{thebibliography}

\end{document}